\documentclass{article}
\usepackage[utf8]{inputenc}
\usepackage{amsmath}
\usepackage{amsthm}
\usepackage{amsfonts}
\usepackage{hyperref}
\usepackage{xcolor}
\hypersetup{
    colorlinks,
    linkcolor={red!80!black},
    citecolor={blue!80!black},
    urlcolor={blue!80!black}
}
\usepackage{float}
\usepackage[margin=1in]{geometry}
\usepackage{graphicx}
\usepackage{subcaption}
\usepackage{caption}
\usepackage{url}

\newcommand{\order}{{\mathcal O}}
\newcommand{\ulambda}{q_0}
\DeclareMathOperator{\sign}{sign}
\DeclareMathOperator{\sech}{sech}

\newcommand{\uhat}{\hat{u}}
\newcommand{\qhat}{\hat{q}}
\newcommand{\Ahat}{\hat{A}}
\newcommand{\Bhat}{\hat{B}}
\newcommand{\FF}{\mathcal{F}}
\newcommand{\real}{\mathbb{R}}

\newtheorem{thm}{Theorem}
\newtheorem{lemma}{Lemma}

\title{Approximation of arbitrarily high-order PDEs by first-order hyperbolic relaxation}
\author{David I. Ketcheson\thanks{Computer, Electrical, and Mathematical Sciences \& Engineering Division, 
King Abdullah University of Science and Technology, Thuwal 23955, Saudi Arabia, (\url{david.ketcheson@kaust.edu.sa})} 
\and Abhijit Biswas\thanks{Computer, Electrical, and Mathematical Sciences \& Engineering Division, King Abdullah 
University of Science and Technology, Thuwal 23955, Saudi Arabia, (\url{abhijit.biswas@kaust.edu.sa})}}

\begin{document}
\maketitle

\abstract{
    We present a framework for constructing a first-order hyperbolic system
    whose solution approximates that of a desired higher-order evolution equation.
    Constructions of this kind have received increasing interest in recent
    years, and are potentially useful as analytical or computational
    tools for understanding the corresponding higher-order equation.
    We perform a systematic analysis of a family of linear model equations
    and show that for each member of this family there is a stable hyperbolic
    approximation whose solution converges to that of the model equation
    in a certain limit.
    We then show through several examples that this approach can be applied 
    successfully to a very wide range of nonlinear PDEs of practical interest.
}

\section{Introduction}
Partial differential equations (PDEs) can be broadly grouped into important classes
(e.g. hyperbolic, parabolic, elliptic) that share certain mathematical properties.
Equations in different classes require different methods of analysis and numerical
solution.  In some cases it is known to be possible to approximate the solutions
of a PDE in one class by solutions of a PDE in a different class.  For instance,
Cattaneo and Vernotte \cite{cattaneo1958forme,vernotte1958paradoxes},
constructed a hyperbolic approximation (or \emph{relaxation}) of the heat equation.
While relatively little work was done on such approximations in subsequent decades,
they have lately received a great deal of attention 
\cite{antuono2009dispersive,grosso2010dispersive,toro2014advection,mazaheri2016first,
favrie2017rapid,ruter2018hyperbolic,li2018new,dhaouadi2019extended,escalante2019efficient,
chesnokov2019hyperbolic,bassi2020hyperbolic,gavrilyuk2022hyperbolic,besse2022perfectly,
dhaouadi2022hyperbolic,dhaouadi2022first,chesnokov2023strongly,gavrilyuk2024conduit}.
In each of the works just cited, the authors develop a hyperbolic approximation for
a specific higher-order PDE, often with the goal of removing numerical stiffness.
The techniques used in developing these equations appear to be specific to each individual equation.
The spatial order of all the PDEs studied so far in this way is at most three.

In this work we present a general technique for constructing a system of first-order hyperbolic PDEs
(referred to as a hyperbolic relaxation, or \emph{hyperbolization})
whose solution \emph{approximates} that of a given system of higher-order
evolution PDEs.  We begin by presenting some motivating examples in Sections \ref{sec:example-heat}-\ref{sec:example-kdv}.
We present these known examples in a way that suggests a more general approach
that unifies some of the existing work in this area and provides a hyperbolization
for equations of arbitrary order.
A key question that arises is the well-posedness of the Cauchy problem for
the resulting hyperbolic system (i.e., the stability of the dispersion relation).
In Section~\ref{sec:general} we study a family of linear PDEs of arbitrary order, for
which we derive the necessary and sufficient conditions for
stability of the hyperbolization.  This leads to a unique hyperbolization for any high-order 
evolution PDE, which is provably stable for linear PDEs in the family we study.  
In Section~\ref{sec:convergence} we prove a convergence result relating the solution of 
the stable hyperbolic relaxation constructed in Section~\ref{sec:general} and the original
high-order PDE.
We conjecture that the general hyperbolization we have constructed is also stable
for other high-order linear and nonlinear PDEs (whenever the original equation is stable).
In Section~\ref{sec:examples} we provide some examples supporting and extending this conjecture
by applying our approach to various nonlinear PDEs, including complex-valued 
equations, systems of equations, equations with mixed derivatives, and equations with up
to (at least) 4th-order derivatives.  We conclude in Section~\ref{sec:conclusion} with a
discussion of important open questions.

We emphasize that the main goal of the present work is not to provide further
motivation for hyperbolization, but to dramatically increase its range of applicability and improve
our understanding of it by tying together the numerous existing ad hoc examples through a unified
general approach.

\subsection{Prior work and motivation}
Let us review in more detail the existing literature on hyperbolization and
the proposed motivation for its use.
Toro \& Montecinos \cite{toro2014advection} gave a general approach to and analysis of
hyperbolization of second-order advection-reaction-diffusion equations, including a detailed
study of the efficiency of the numerical solution of the hyperbolized model compared to
the original model.  
%They found that hyperbolization provided a substantial advantage
%in combination with coarse spatial grids, but on sufficiently fine grids became less
%efficient than solving the original model.
Mazaheri et. al. \cite{mazaheri2016first} provided a general approach to hyperbolization of
scalar PDEs that include first-, second-, and third-order spatial derivatives (i.e.,
advection-diffusion-dispersion equations), and demonstrated it through application to the
KdV and Burgers equations.  Many other hyperbolized versions of PDE models have
been proposed, e.g. for Korteweg-de Vries (KdV) \cite{besse2022perfectly},
Benjamin-Bona-Mahoney (BBM) \cite{gavrilyuk2022hyperbolic},
Serre-Green-Naghdi (SGN) \cite{favrie2017rapid,escalante2019efficient,bassi2020hyperbolic},
multilayer dispersive shallow water equations \cite{chesnokov2019hyperbolic},
the compressible Navier-Stokes equations \cite{li2018new},
the Euler-Korteweg equation (in hydrodynamic form) \cite{dhaouadi2019extended},
and a range of additional dispersive wave equations 
\cite{antuono2009dispersive,grosso2010dispersive,dhaouadi2022hyperbolic,dhaouadi2022first,chesnokov2023strongly,gavrilyuk2024conduit},
as well as elliptic equations \cite{ruter2018hyperbolic}.

%These hyperbolic relaxations can be viewed in relation to other existing ideas,
%such as artificial compressibility...

%Nevertheless, it seems that so far there is no general approach to apply these ideas to arbitrary PDEs in
%a systematic way.
%On the contrary, existing presentations of what we will call \emph{hyperbolization} are based on a variety of
%ad hoc techniques or are given with little explanation.
%Here we develop a unifying approach, illustrated with several examples,
%and investigate important questions that arise naturally once the general approach is laid out.

Most of the foregoing works explore hyperbolization as a means to remove
numerical stiffness.
For an evolution PDE with spatial derivatives of order $m$, stability of any explicit numerical method
requires a time step $\Delta t = \order(\Delta x^m)$, and typically for $m>1$ it is best
to use implicit numerical solvers.  For PDEs with mixed space and time derivatives,
a numerical solution is usually obtained by performing an elliptic solve at each step.  In all of
these cases, the complexity and per-step cost of numerical solvers is substantially increased compared
to that of first-order hyperbolic systems, which can typically be solved efficiently using 
only explicit methods.  Replacing the stiff high-order PDE by a first order PDE is therefore
quite appealing.
However, the hyperbolic approximations referenced above all have the property that they
approximate the higher-order system accurately only when certain characteristic speeds
become large.  The stiffness of the original problem can thus return through the mechanism
of these large characteristic speeds,
so it is not clear \emph{a priori} that solution of a hyperbolic approximation
is more efficient.  Many existing works on hyperbolization do not consider this question.
The most thorough theoretical comparison of efficiency in this regard shows that (for certain second-order
equations) a hyperbolized approximation can be solved more efficiently than the original problem 
when the spatial grid is relatively coarse,
but it becomes inefficient on finer grids \cite{toro2014advection}.
We consider the relative computational efficiency of hyperbolization plus discretization versus
direct high-order PDE discretization to be an important matter for future work.

Another potential advantage of hyperbolic formulations is more straightforward: imposition of boundary conditions --
especially non-reflecting boundary conditions,
which are more developed for hyperbolic problems than for general PDEs.  For instance, in
\cite{besse2022perfectly}, some dispersive wave equations are approximated by hyperbolic systems
in order to formulate absorbing boundary conditions using perfectly matched layers.
As we will see, hyperbolized PDEs do require additional initial conditions, but these
can usually be obtained in a very natural way and are similar to those required for
PDEs with multiple temporal derivatives.

Finally, a fundamental motivation for approximation by hyperbolic equations is that of causality.
PDEs with higher-order terms allow for action at a distance or arbitrarily fast propagation
of high-wavenumber perturbations, and are thus incompatible with the theory of relativity.
In contrast, for hyperbolic PDEs perturbations travel at or
below some maximum speed.  This served as the motivation for some of the early work on
hyperbolic relaxation of the heat equation.

The approach introduced by Jin and Xin \cite{jin1995relaxation,liu2001basic} is also referred to
as \emph{hyperbolic relaxation}, but it differs fundamentally from the concept dicsussed
herein.
In Jin-Xin relaxation, the relaxation terms are usually
algebraic, the starting point is a first-order nonlinear hyperbolic system, and the result is
a dissipative approximation.  Here we start from higher-order (not hyperbolic)
PDEs, we use relaxation terms that involve differential operators, and the resulting
approximation is non-dissipative if the original problem is non-dissipative.
See Section \ref{sec:jinxin} for discussion of one connection between these ideas.

Before entering into further details, we present two introductory examples that have
appeared before and serve to motivate the technique developed later in the present work.

\subsection{The Heat Equation} \label{sec:example-heat}
The oldest example of hyperbolization of which we are aware was proposed independently by
Cattaneo \cite{cattaneo1958forme} and Vernotte \cite{vernotte1958paradoxes}.
The heat equation
\begin{align} \label{heat}
    \partial_t u & = \partial_{x}^2 u
\end{align}
can be written superficially as a first-order system by introducing the auxiliary variable $v(x,t) = u_x$:
\begin{subequations}
\begin{align}
    \partial_t u & = \partial_x v \\
    \partial_x u & = v. \label{uv-constraint}
\end{align}
\end{subequations}
The hyperbolic formulation is obtained by relaxing the constraint \eqref{uv-constraint}
as part of an evolution equation for $v$:
\begin{subequations} \label{heatH}
\begin{align}
    \partial_t u & = \partial_x v \\
    \tau \partial_t v & = (\partial_x u - v).
\end{align}
\end{subequations}
The first equation comes directly from \eqref{heat}, while the second
equation causes $v(x,t)$ to relax toward $\partial_x u$.  The parameter
$\tau>0$ controls the scale of this relaxation time; as $\tau \to 0$,
one expects that the solution of \eqref{heatH}
will tend to that of \eqref{heat}.  Defining $q=[u,v]^T$, system \eqref{heatH} can be 
written in the form
\begin{align} \label{linear-hyperbolic}
    \partial_t q + A \partial_x q = Bq
\end{align}
where $A$ is diagonalizable with eigenvalues $\pm \sqrt{\tau^{-1}}$, so this system is hyperbolic.
If we assume a solution of the form $q(x,t)=\exp(i(kx-\omega t))$, we find that
the full dispersion relation for \eqref{heatH} is
\begin{align}
    \tau \omega^2 + i \omega - k^2 & = 0,
\end{align}
resulting in
\begin{align}
    \omega_\pm(k) & = \frac{i}{2\tau}\left(-1 \pm \sqrt{1-4k^2\tau}\right) \\
              & = \frac{i}{2\tau}\left(-1 \pm (1-2k^2\tau)\right) + \order(k^4\tau).
\end{align}
We see that $\omega_+(k)\approx -i k^2$, so that the dispersion relation
of \eqref{heat} is approximately preserved.  The accuracy of the approximation is seen to improve as $\tau \to 0$.  
Notice that $\omega_-$ has negative real part proportional to $\tau^{-1}$,
so there is a spurious mode but it decays quickly if $\tau$ is small.
Note also that if $4k^2 \tau > 1$, then $\omega$
has a real part and a negative imaginary part.  This means that solutions are stable but
we should expect to see wave-like behavior for high wavenumbers or when $\tau$ is large.
%Note that if we take $\tau<0$, there exist exponentially growing solutions. 
%\abhi{Maybe we can omit this note since it confused the first referee and is already obvious.}

%One could also view $\lambda = 1/\tau$ as a Lagrange multiplier, and the second equation of
%\eqref{heatH} as a way of softly enforcing the constraint $v=u_x$.  It is important to
%note that this multiplier must be positive.  If we replace $\lambda$ by $-\lambda$,
%then the characteristic speeds of the first-order system \eqref{linear-hyperbolic} are
%not real and the resulting system of equations is unstable.

%The system \eqref{heatH} has been employed in some works based on arguments that
%it is more physically reasonable than the heat equation \eqref{heat}.  We will not discuss
%this question here as we are primarily interested in systems like \eqref{heatH} as approximations
%to higher-order PDEs.

Figure \ref{fig:heat} shows a comparison of the solution of the heat equation with that of
its hyperbolic approximation, with a Gaussian initial condition.  For large values of 
$\tau$, the solution behaves similarly to that of the wave equation, with the initial
pulse breaking into two waves propagating in opposite directions, consistent with the dispersion
relation analysis above.  For smaller values of $\tau$, the solution behaves very similarly to that of the heat equation.
We remark that wave-like heat transport under certain conditions has long been theorized and
recently been experimentally observed \cite{chester1963second,yan2024thermography}.

\begin{figure}
    \center
    \includegraphics[width=6in]{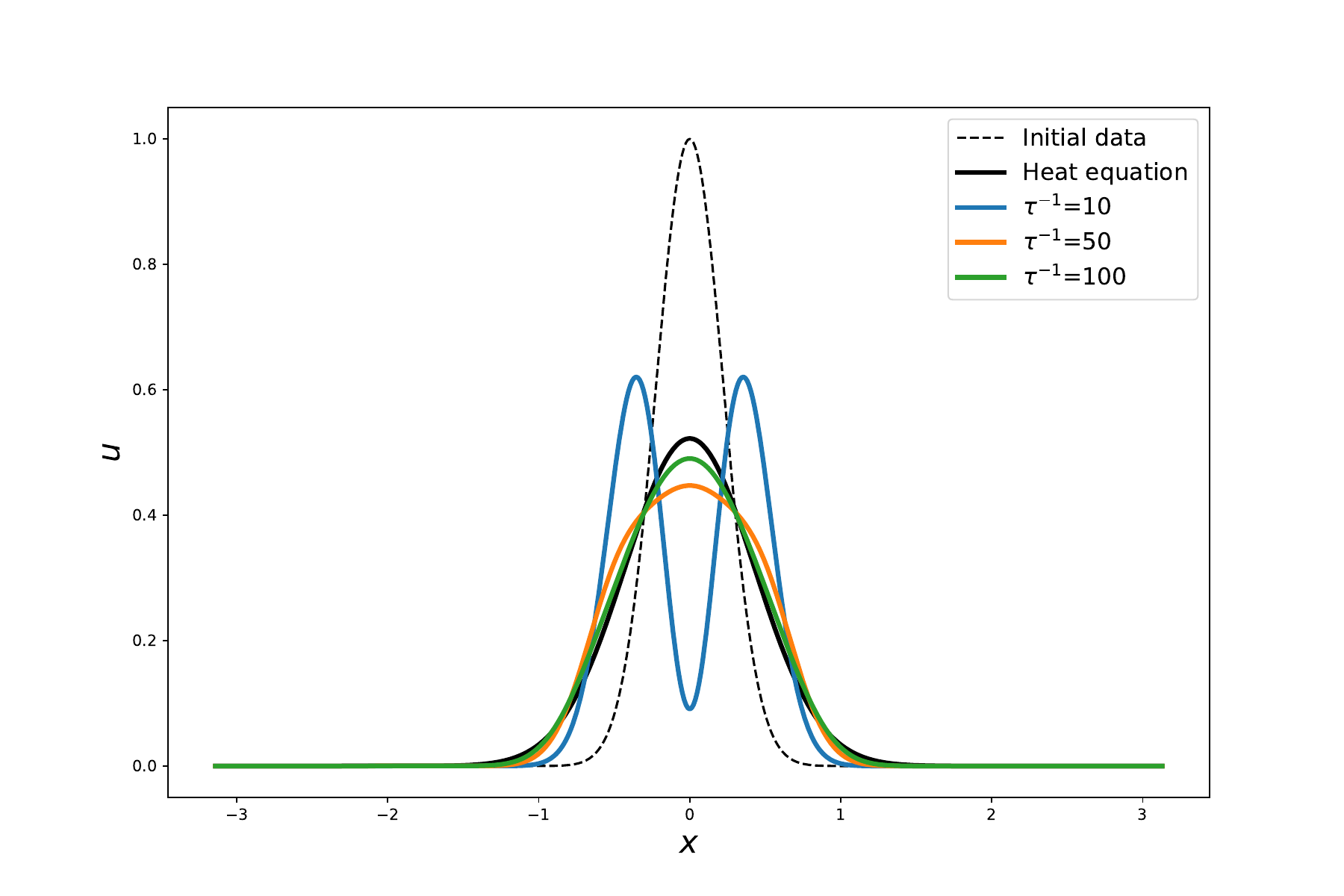}
    \caption{Comparison of the solution of the heat equation \eqref{heat} and its hyperbolic
    approximation \eqref{heatH}.  The approximation improves with smaller values of $\tau$.}
    \label{fig:heat}
\end{figure}

\subsubsection{Relation to Jin-Xin Relaxation} \label{sec:jinxin}
Jin \& Xin \cite{jin1995relaxation} introduced a technique for approximating a hyperbolic PDE
\begin{align*}
    \partial_t u + \partial_x f(u) & = 0
\end{align*}
by the hyperbolic system
\begin{subequations} \label{jinxin}
\begin{align}
    \partial_t u + \partial_x v & = 0 \\
    \partial_t v + a \partial_x u & = -\frac{1}{\tau} (v- f(u)). 
\end{align}
\end{subequations}
We can formally cast the hyperbolic relaxation of the heat equation in this
form by writing $f(u) = -\partial_x u$ and taking $a=0$. The result is equivalent
to the system \eqref{heatH} (though $v$ is defined with the opposite sign).
In \cite{jin1995relaxation} it was shown that the solution of \eqref{jinxin}
is approximated by that of
$$
    \partial_t u + \partial_x f(u) = \tau \partial_x\left( (a - (f'(u))^2) \partial_x u \right),
$$
and as a result the stability of \eqref{jinxin} requires the \emph{sub-characteristic condition}
$|f'(u)| \le \sqrt{a}$.
Taking the same approach for the heat equation and including the term $a \partial_x u$ in the
evolution equation for $v$, one finds that the solution of the resulting relaxation system is
approximated (to first order in $\tau$) by
\begin{align} \label{heat-jin-xin}
    \partial_t u & = \partial_x^2 u + \tau\left(a - \partial_x^2\right) \partial_x^2 u.
\end{align}
This equation is stable for any non-negative $a$ (and in fact for any $a\ge -\tau^{-1}$),
so in particular one can simply take $a=0$ as we have done above.

For hyperbolic relaxation of higher-order PDEs (such as that in the next section),
there is not a straightforward relationship to Jin-Xin relaxation.

\subsection{The Korteweg-de Vries Equation} \label{sec:example-kdv}
Next, consider the Korteweg-de Vries (KdV) equation:
\begin{align} \label{kdv}
    \partial_t u + u\partial_x u +\partial^3_x u & = 0.
\end{align}
We introduce $v \approx \partial_x u$ and $w\approx \partial_x v$, so that $\partial_x w \approx \partial^3_x u$.
We use $w$ to write a first-order approximation of \eqref{kdv}, and
introduce evolution equations for $v, w$ that incorporate relaxation terms that
tend to enforce the foregoing approximations
\begin{subequations}\label{kdvH}
\begin{align} 
\partial_t u + u\partial_x u + \partial_x w & = 0 \\
\tau \partial_t v & = (\partial_x v - w) \label{kdvHb} \\
\tau \partial_t w & = -(\partial_x u - v). \label{kdvHc}
\end{align}
\end{subequations}
As before, we require $\tau>0$ and we expect that the solution of \eqref{kdvH} approaches that of \eqref{kdv}
as $\tau \to 0$.  The system \eqref{kdvH} was introduced previously as a means to
implement non-reflecting numerical boundary conditions \cite[Eqn. (13)]{besse2022perfectly}.
A comparison of the solutions of the KdV equation \eqref{kdv} and its hyperbolic
approximation \eqref{kdvH} for two different values of $\tau$ is shown
in Figure \ref{fig:kdv}.

\begin{figure}
    \center
    \includegraphics[width=6in]{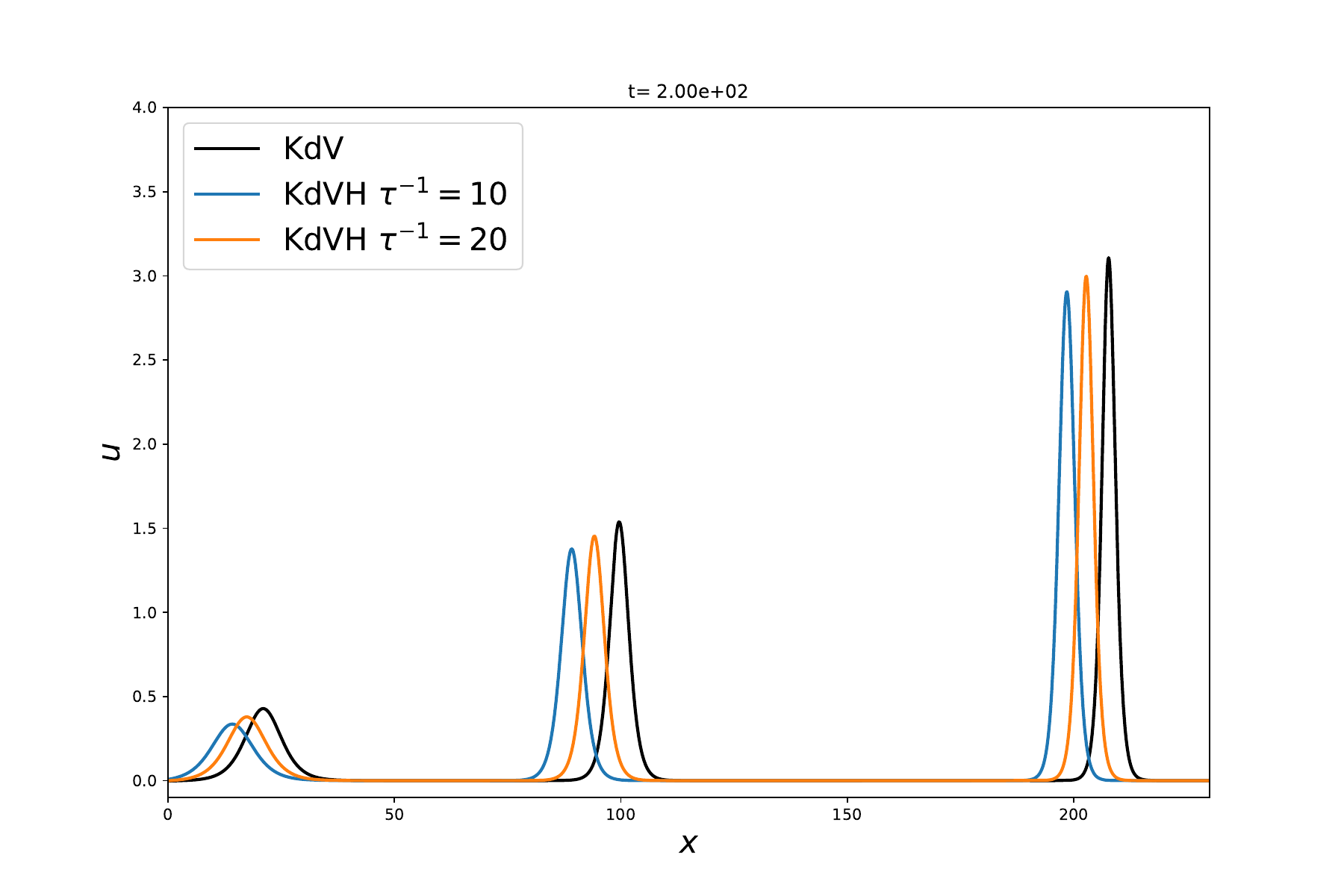}
    \caption{Comparison of the solution of the KdV equation \eqref{kdv} and its hyperbolic
    approximation \eqref{kdvH}.  The approximation improves with smaller values of $\tau$.}
    \label{fig:kdv}
\end{figure}

Although the formulation of \eqref{kdvH} seems fairly natural based on
our discussion and the similar handling of the heat equation above,
it is worth noting that it involves certain choices that at this point might
seem arbitrary.
Why are the right-hand sides of \eqref{kdvHb} and \eqref{kdvHc} chosen as they
are, rather than the reverse?  Why are the signs of those terms chosen as they are?
How would the properties of the system change if these signs or order were modified?
In the next section, we will answer these questions and devise a general approach
to formulating similar hyperbolizations.
%Neither these questions, nor the general approach to formulating similar hyperbolic
%approximations, has been discussed in the literature.

The main contributions of the present work are:
\begin{itemize}
    \item construction of stable hyperbolic approximations to a family of arbitrarily high-order PDEs;
    \item proof of pointwise convergence of the hyperbolic relaxation in Fourier space;
    \item successful application of the technique to multiple problems outside the scope of what has been done before.
\end{itemize}

%\subsection{Outline}
%As the foregoing examples suggest, many higher-order PDEs can be approximated by systems
%of first-order hyperbolic PDEs.  However, care is required in order to obtain a hyperbolization
%that is stable.  In the remainder of this work we develop an approach that
%provides a prescription to obtain stable hyperbolic approximations for a broad class of PDEs.
%In Section \ref{sec:general} we consider linear evolution equations with an arbitarily high-order
%spatial derivative and show that, while almost all of the apparent hyperbolic approximations are
%unstable, there is always a stable hyperbolization.
%In Section \ref{sec:3} we briefly examine the accuracy of the hyperbolized model and
%its discretization.  In Section \ref{sec:examples} we consider several examples of higher-order
%PDEs, including complex exquations, systems, and equations with up to 4th-order derivatives.
%In each case, we show that a stable hyperbolization can be obtained based on the theory
%of Section \ref{sec:general}.  We conclude in Section \ref{sec:conclusion} with a discussion of
%the significance of hyperbolization and a number of open questions.

\section{A general approach to hyperbolization} \label{sec:general}

In this section we provide an approach to hyperbolization that generalizes
the examples above and other examples in the literature.
Let us consider the Cauchy problem for the linear model equation
\begin{subequations} \label{linear2}
\begin{align}
    \partial_t u + \sigma_0 \partial_x^m u & = 0 \label{linear2a} \\
    u(x,0) & = u_0(x) \\
    \textrm{where } \sigma_0 & = \begin{cases} \pm 1 & m \textrm{ odd} \\ (-1)^{m/2} & m \textrm{ even}. \end{cases} \label{sigma-cond}
\end{align}
\end{subequations}
The restriction on the sign $\sigma_0$ ensures that all solutions remain bounded
for $t>0$.
%where $\sigma_0=\pm 1$.  For odd values of $m$, either sign is permissible while 
%for even values of $m$ only $\sigma_0=i^m$ is stable.
We introduce new variables $q_j \approx \partial_x^j u$ for $j = 0, 1, 2, \dots, m-1$.
Then we can approximate \eqref{linear2} by the system of equations
\begin{subequations} \label{gen-hyp}
\begin{align}
    \partial_t q_0 + \sigma_0 \partial_x q_{m-1} & = 0 \\
    \tau \partial_t q_j & = \sigma_j (q_{i_j} - \partial_x q_{i_j-1}) & j = 1, 2, \dots, m-1.
\end{align}
\end{subequations}
Here we take $\tau>0$, $\sigma_j = \pm 1$, and the sequence $i_1, i_2,
\dots, i_{m-1}$ is a permutation of the integers from 1 to $m-1$.
The idea is to introduce an evolution equation for each $q_j$, and incorporate in each equation
one of the constraints $q_k \approx  \partial_x q_{k-1} \approx \partial_x^k u$.  We now investigate
how one should choose
which constraint appears in which evolution equation (i.e., we determine the permutation sequence $i$) and the
signs $\sigma_j$ of the constraint multipliers.

For instance, consider the linearized KdV equation:
\begin{align} \label{linear-kdv}
    \partial_t u + \partial_x^3 u & = 0.
\end{align}
We have $m=3$, and the system \eqref{gen-hyp} takes one of the two following forms:
\begin{subequations} \label{kdva}
\begin{align}
    \partial_t q_0 + \partial_x q_2 & = 0 \\
    \tau \partial_t q_1 & = \pm (q_1 - \partial_x q_0) \\
    \tau \partial_t q_2 & = \pm (q_2 - \partial_x q_1),
\end{align}
\end{subequations}
or
\begin{subequations} \label{kdvb}
\begin{align}
    \partial_t q_0 + \partial_x q_2 & = 0 \\
    \tau \partial_t q_1 & = \pm (q_2 - \partial_x q_1) \\
    \tau \partial_t q_2 & = \pm (q_1 - \partial_x q_0).
\end{align}
\end{subequations}
Altogether the above equations represent 8 possible choices.
The question immediately arises: which of these systems (if any) is preferable?
At a glance, the most natural choice is \eqref{kdva} with "-" on the right
side of the second and third equation, since then these equations directly impose
that $q_1$ and $q_2$ relax to $\partial_x q_0$ and $\partial_x q_1$.  However,
that system is in fact \emph{unstable}.
As we will see, only one of these 8 systems is stable, and in fact for each
value of $m$, there is a unique choice of permutation and signs in \eqref{gen-hyp}
such that the resulting approximation of \eqref{linear2} is stable.

\subsection{Necessary Conditions for Stability} \label{sec:stability}
We can write \eqref{gen-hyp} in matrix form as
\begin{subequations} \label{matform}
\begin{align}
    D \partial_t q + A \partial_x q & = B q.
\end{align}
where
\begin{align} \label{ABD}
    A & = \begin{bmatrix}  0 & \sigma_0 \\ P & 0 \end{bmatrix} & B & = \begin{bmatrix} 0 & 0 \\ 0 & P \end{bmatrix}, 
    & D = \begin{bmatrix} 1 & & & \\ & \tau & & \\ & & \ddots & \\ & & & \tau \end{bmatrix},
\end{align}
\end{subequations}
and $P$ is a signed permutation matrix\footnote{A (real) signed permutation matrix is square with all entries equal to -1, 0 or +1,
and has exactly one non-zero entry in each row and each column.} of size $m-1$.
Applying the ansatz $q(x,t) = q_0 \exp(i(kx-\omega t))$, we find
that the dispersion relation is given by the solution of
\begin{align} \label{dispreldet}
    \det(-i \omega D + ik A - B) & = 0.
\end{align}
Since the time-dependence of the ansatz has the form $\exp(- i \omega t)$,
the solution is stable only if the imaginary part of $\omega$ is non-positive.
It is convenient to consider two necessary conditions for stability: stability for $k=0$, and
stability for large $|k|$.  It turns out that it is sufficient to consider the special case
$\tau = 1$, in which case we can write \eqref{dispreldet} as
\begin{align} \label{dispreldet-lam1}
    \det(\omega I - (k A + iB)) & = 0.
\end{align}

\subsubsection{Low-wavenumber stability}
A necessary condition for stability at low wavenumbers is stability for $k=0$.
In that case, the dispersion relation \eqref{dispreldet-lam1} reduces to
\begin{align}
    \det(\omega I - i B) & = 0,
\end{align}
so stability is obtained only if all eigenvalues of $B$ lie in the closed left half-plane.
The matrix $B$ has one eigenvalue equal to zero and the rest equal to the eigenvalues of $P$,
so we must choose $P$ such that its eigenvalues are in the closed left half-plane.
The following lemma characterizes such permutation matrices.

\begin{lemma} \label{lem:P}
    Let $P$ be a real signed permutation matrix with all eigenvalues in the closed left
    half-plane.  Then
    \begin{align} \label{small-k}
        P = R-D
    \end{align}
    where $R$ is skew-symmetric ($R^T = -R$), and $D$ is a diagonal matrix with all entries equal to 1 or 0.
\end{lemma}
\begin{proof}
Any permutation can be decomposed in terms of disjoint cycles, and the set of eigenvalues of the
permutation is just the union of the eigenvalues of each cycle.  The eigenvalues of a cycle
of length $n$ are just the $n$th roots of unity.  The same can be said for \emph{signed} permutation
matrices, except that the eigenvalues of a signed cycle of length $n$ are in general $n$th roots 
of either $+1$ or $-1$.  The only sets of roots of $\pm1$ that lie fully in the closed left
half-plane are $-1$ itself and the square roots of $-1$.  Therefore, $P$ must be the product of
cycles of length 1 and/or 2 only, and the 2-cycle components must be antisymmetric.
\end{proof}

\subsubsection{High-wavenumber stability}
Dividing \eqref{dispreldet-lam1} by $k^m$ and taking the limit $k\to \infty$, we obtain the
high-wavenumber limiting dispersion relation
\begin{align}
    \det(c(k) I - A) & = 0,
\end{align}
where $c(k)=\omega(k)/k$ is the phase velocity.  This is just
the dispersion relation of the homogeneous hyperbolic system $q_t + Aq_x$, whose solutions are
stable if and only if $A$ has only real eigenvalues.  
\begin{lemma} \label{lem:Areal}
    Let $P$ be a real signed permutation matrix and let $A$ be given by \eqref{ABD}
    with $\sigma_0=\pm1$.  If $A$ has only real eigenvalues then
    \begin{align}\label{large-k}
        A = S + D
    \end{align}
    where $S$ is symmetric and $D$ is a
    %$\footnote{An anti-diagonal matrix is square and has all entries equal to zero except those on the diagonal going from lower left to upper right.}
    diagonal matrix with all entries equal to $\pm 1$ or 0.
\end{lemma}
\begin{proof}
    The proof is similar to that of Lemma \ref{lem:P}.  Notice that $A$ itself is a
    signed permutation matrix, so to have only real eigenvalues it must consist
    only of 2-cycles with equal signs and/or 1-cycles.
\end{proof}
It is possible to prove an even stronger result: $A$ must be anti-diagonal (with non-zero
entries $a_{ij}$ only for $i+j=m+1$).  But Lemma \ref{lem:Areal} is sufficient to prove
the main result in the next section.

\subsection{The unique stable hyperbolization}
Let us continue our example above with the linearized KdV equation \eqref{linear-kdv}.
Now $P$ is a $2 \times 2$ matrix and from Lemma \ref{lem:P} we see that either
\begin{align}
P & = -I & \text{ or } P & = \pm \begin{bmatrix} 0 & 1 \\ -1 & 0 \end{bmatrix}.
\end{align}
From the stability condition \eqref{large-k} for large $|k|$, we see that the only possibility is
\begin{align}
P & = \begin{bmatrix} 0 & -1 \\ 1 & 0 \end{bmatrix},
\end{align}
yielding the hyperbolized system (corresponding to \eqref{kdvb} with minus in the second equation
and plus in the third)
\begin{align*}
    \partial_t q_0 + \partial_x q_2 & = 0 \\
    \tau \partial_t q_1  - \partial_x q_1 & = - q_2 \\
    \tau \partial_t q_2 + \partial_x q_0 & = q_1.
\end{align*}

Returning to the general equation \eqref{linear2}, since $a_{1m}=\pm 1$, condition \eqref{large-k} implies
$p_{m-1,1}=\pm 1$.  Then condition \eqref{small-k} implies $p_{1,m-1}=\mp 1$; hence
\eqref{large-k} implies $p_{m-2,2}=\mp1$, and so forth.  Continuing in this
manner, it can be shown that the only way to satisfy the necessary stability conditions
is to choose $P$ (and $A$) to be anti-diagonal, with entries
\begin{align} \label{stableP-plus}
    p_{ij} & = \begin{cases} 0 & i+j \ne m \\
                            \sigma_0 (-1)^{j-1} & i+j = m \ \ j \le m/2 , \\
                            \sigma_0 (-1)^{m-j} & i+j = m \ \ j > m/2.
    \end{cases}
\end{align}

These necessary conditions, which we obtained by considering the stability for both low and 
high wave numbers, turn out to be sufficient for the stability of the system \eqref{matform} for all
wave numbers $k$ and all  $\tau > 0$ . The result is summarized in the following theorem.
\begin{thm} \label{thm:stability}
    The system \eqref{matform} is stable for all wavenumbers $k$ and all $\tau>0$ if and
    only if $P$ is given by \eqref{stableP-plus}.
\end{thm}
\begin{proof}
The ``only if" part has been proven already by consideration of the necessary conditions
above.  It remains to prove sufficiency.

From the general dispersion relation \eqref{dispreldet} we see that $\omega$ is
given by the eigenvalues of $\Lambda(kA+iB)$, where $\Lambda = D^{-1}$.
Since $A$ is symmetric and $B$ is skew-symmetric, $C=kA+iB$ is Hermitian.  Furthermore,
since $\tau >0$ we can write
$$
    \Lambda C = \Lambda^{1/2}(\Lambda^{1/2} C \Lambda^{1/2}) \Lambda^{-1/2},
$$
from which we see that $\Lambda C$ is similar to a Hermitian matrix, hence it has only real
eigenvalues.  Since $\omega$ is real, the hyperbolized system is stable for all $k, \tau$.
\end{proof}
Note that in the limit $\tau = 0$ , the relaxation system \eqref{matform} simply reduces 
to the original PDE \eqref{linear2}.

The unique stable hyperbolic system, combining \eqref{matform} with \eqref{stableP-plus}, is
\begin{subequations} \label{uni-hyp}
\begin{align}
    \partial_t q_0 + \sigma_0 \partial_x q_{m-1} & = 0 \\
    \tau \partial_t q_j & = \sigma_0 (-1)^j (q_{m-j} - \partial_x q_{m-j-1}) & j = 1, 2, \dots, \left\lceil \frac{m}{2} \right\rceil \\
    \tau \partial_t q_j & = \sigma_0 (-1)^{m-j-1} (q_{m-j} - \partial_x q_{m-j-1}) & j = \left\lceil \frac{m}{2} \right\rceil+1, \dots, m-1.
\end{align}
\end{subequations}
Note that \eqref{uni-hyp} is stable for $\sigma_0$ satisfying \eqref{sigma-cond}.
It is not stable for even $m$ with $\sigma_0=-i^m$,
since the model equation \eqref{linear2a} is unstable in that case.

In order to approximate the solution of an initial value problem for equation \eqref{linear2},
one must also supply initial data for $q_0, q_1, \dots, q_{m-1}$.  The natural choice is
\begin{align} \label{initial-values}
    q_j(x,t=0) & = \partial_x^j u(x,t=0).
\end{align}

\subsection{General Linear Scalar Evolution PDEs}

Next we consider the more general model equation
\begin{align} \label{general-linear}
    u_t + \sum_{j=0}^{m-1} \alpha_j \partial_x^j u + \sigma_0 \partial^m_x u & = 0,
\end{align}
where we assume the values $\alpha_j$ are such that solutions remain bounded; in 
particular this implies that $\alpha_0\ge 0$ 
and if $m$ is even then $\sign(\alpha_m)=(-1)^{m/2}$.
There are many ways to modify \eqref{uni-hyp} to account for
the lower-order derivative terms present in \eqref{general-linear},
since each term $\partial_x^j u$ could be replaced by either $q_j$ or
$\partial_x q_{j-1}$.  Taking always the first choice yields
\begin{subequations} \label{gen-hyp-B}
\begin{align}
    \partial_t q_0  + \sigma_0 \partial_x q_{m-1} & = - \sum_{j=0}^{m-1} \alpha_j q_j \\
    \tau \partial_t q_j & = \sigma_j (q_{i_j} - \partial_x q_{i_j-1}) & j = 1, 2, \dots, m-1,
\end{align}
\end{subequations}
while taking always the second choice yields
\begin{subequations} \label{gen-hyp-A}
\begin{align}
    \partial_t q_0  + \sum_{j=1}^{m-1} \alpha_j \partial_x^j q_{j-1} + \sigma_0 \partial_x^m q_{m-1} & = - \alpha_0 q_0 \\
    \tau \partial_t q_j & = \sigma_j (q_{i_j} - \partial_x q_{i_j-1}) & j = 1, 2, \dots, m-1.
\end{align}
\end{subequations}

%\begin{subequations} \label{modify-B}
%\begin{align}
%    \partial_t q_0  + \alpha_m \partial_x q_{m-1} & = - \sum_{j=0}^m \alpha_j q_j \\
%    \tau \partial_t q_j & = \sigma_0 (-1)^j (q_{m-j} - \partial_x q_{m-j-1}) & j = 1, 2, \dots, \left\lceil \frac{m}{2} \right\rceil \\
%    \tau \partial_t q_j & = \sigma_0 (-1)^{m-j-1} (q_{m-j} - \partial_x q_{m-j-1}) & j = \left\lceil \frac{m}{2} \right\rceil+1, \dots, m-1.
%\end{align}
%\end{subequations}
%or modify $A$ (although the non-differentiated term must still be incorporated into $B$):
In the following analysis, and in the examples in Section \ref{sec:examples},
we follow \eqref{gen-hyp-A}.  It is an open question whether there is some advantage to
using \eqref{gen-hyp-B}, or some combination of the two.

\begin{thm} \label{thm:gl-stability}
For the hyperbolization \eqref{gen-hyp-A}, condition \eqref{small-k} is necessary and sufficient for
stability when $k=0$.  Choosing $P$ as in \eqref{stableP-plus} is sufficient for stability as $|k|\to\infty$.
\end{thm}
\begin{proof}
    Consider the hyperbolization of \eqref{general-linear} given by \eqref{modify-A}.
    This system can be written in the form \eqref{matform} but with $A$ and $B$ replaced by
    \begin{align}
    \Ahat & = \begin{bmatrix}  \alpha & \sigma_0 \\ P & 0 \end{bmatrix} & \Bhat & = \begin{bmatrix} -\alpha_0 & 0 \\ 0 & P \end{bmatrix}.
    \end{align}
    For $k=0$, stability requires that the eigenvalues of $\Bhat$ lie in the left half-plane.
    The eigenvalues of $\Bhat$ are the same as those of $B$, but with the zero eigenvalue
    replaced by $-\alpha_0$.  Hence it is necessary that the eigenvalues of $B$ lie in the
    closed left half-plane and that $\alpha_0\ge 0$; the latter condition is required
    for stability of the original PDE \eqref{general-linear}.

    For large $k$, stability requires that the eigenvalues of $\Ahat$ be real.  
    Using Laplace expansions, it can be shown that if $P$ is given by \eqref{stableP-plus}
    then the characteristic polynomial of $A$ takes the form
    $$
        (\omega-1)^a (\omega+1)^b (\omega^2 - \alpha_1 \omega - |\sigma_0|) = 0.
    $$
    The roots of this polynomial are all real.
\end{proof}

Taking $P$ as in \eqref{stableP-plus} with the hyperbolization \eqref{gen-hyp-A}
yields the system
\begin{subequations} \label{modify-A}
\begin{align}
    \partial_t q_0  + \sum_{j=1}^m \alpha_j \partial_x^j q_{j-1} & = - \alpha_0 q_0 \\
    \tau \partial_t q_j & = \sigma_0 (-1)^j (q_{m-j} - \partial_x q_{m-j-1}) & j = 1, 2, \dots, \left\lceil \frac{m}{2} \right\rceil \\
    \tau \partial_t q_j & = \sigma_0 (-1)^{m-j-1} (q_{m-j} - \partial_x q_{m-j-1}) & j = \left\lceil \frac{m}{2} \right\rceil+1, \dots, m-1.
\end{align}
\end{subequations}
While we have not found a way to generalize Theorem \ref{thm:stability} and
prove stability of \eqref{modify-A} for all wavenumbers, Theorem \ref{thm:gl-stability} suggests
that it is a promising choice.  Computational experiments indicate that this
choice may in fact be stable for all wavenumbers whenever the original problem
\eqref{general-linear} is stable.

The theory developed in this section is based on analysis of linear PDEs only.
In Section \ref{sec:examples} we will apply the approach above to more general
PDEs, including nonlinear equations and equations with mixed space-time derivatives.
For nonlinear problems, one expects that stability of the linearized problem is at
least a necessary condition; this leads again to the choice of signed permutation matrix
determined above.  For problems with mixed space- and time-derivatives, we
will again use linear stability as our guiding principle.  A detailed investigation
of the hyperbolization of such PDEs, and of systems of PDEs, is left to future work.

\section{Accuracy of the hyperbolic approximation}\label{sec:convergence}
As illustrated and discussed above, the solution of the hyperbolized PDE is expected
to converge to that of the original PDE as $\tau\to 0$.  What is the size of
the \emph{hyperbolization error} $\ulambda - u$, and how quickly does it vanish as
$\tau$ decreases?

For the hyperbolized heat equation \eqref{heatH}, using equality of mixed partial derivatives
we find that
\begin{align}
    \partial_t \ulambda & = \partial_x^2 \ulambda - \tau \partial_t^2 \ulambda.
\end{align}
Using equality of mixed partial derivatives again, one obtains
$\partial_t^2 \ulambda = \partial_x^4 \ulambda + \order(\tau)$,
so that (see \eqref{heat-jin-xin})
\begin{align}
    \partial_t \ulambda & = \partial_x^2 \ulambda - \tau \partial_x^4 {\ulambda} + \order(\tau^2).
\end{align}
This suggests that the hyperbolized solution converges linearly to that of the original
equation, and also that one should take
$$
\tau^{-1} > \frac{\max_{x,t} \left| \partial_x^4 u \right|}{\max_{x,t} \left| \partial_x^2 u \right|}
$$
in order to ensure that the error terms are small compared to the leading order terms.
For the initial data shown in Figure \ref{fig:heat}, this suggests one should take approximately $\tau^{-1} > 60$,
which is in general agreement with the results shown.  As time advances, the ratio of derivatives
becomes smaller, and one correspondingly observes improved agreement between the solutions
(even those with large values of $\tau$) after long enough time.

%Similarly, for the NLSh equation \eqref{nlsH}, one can show that
%\begin{align}
%    i \partial_t \ulambda + \partial_x^2 \ulambda + \kappa |\ulambda|^2 \ulambda & = \lambda^{-1} \partial_t^2 \ulambda.
%\end{align}
%Suggesting that also in this case the hyperbolization error is $\order(\lambda^{-1})$.
In the next theorem we reference the following space of functions:
\begin{align}
    \FF_n & = \left\{v(x) = \sum_{j=1}^n \hat{v}_j e^{i k_j x} : \hat{v}, k_j \in \real\right\} & n < \infty.
\end{align}
For the family of linear PDEs \eqref{linear2}, we have the following result.
%one can show that
%$q_0$ satisfies the same equation as $u$ but with a source term of order $\tau$:
%\begin{align}
%    \partial_t \ulambda + \partial_x^m \ulambda & = \order(\tau),
%\end{align}
%which means that the error $E = \ulambda - u$ satisfies
%\begin{align}
%    \partial_t E(x,t) + \partial_x^m E(x,t) & = \order(\tau),
%\end{align}
%suggesting that in this general setting the error is also of order $\tau$.
%This can be proven for the Cauchy problem (initialized using \eqref{initial-values})
%by using the Fourier representation of the solution and linearity.  

\begin{thm} \label{thm:convergence}
Consider the one-dimensional Cauchy problem for \eqref{linear2} with
initial data $u(x,t=0) \in \FF_n$ for some $n$.
Let $q$ denote the solution of the hyperbolization \eqref{uni-hyp}-\eqref{initial-values}, and
let $T$ be given such that $0\le T < \infty$.  Then
\begin{align}
    \| q_0(x,T) - u(x,T) \|_{\infty} = \order(\tau).
\end{align}
\end{thm}
Theorem \ref{thm:convergence} shows that the Fourier transform of the hyperbolized 
solution converges to that of the original problem pointwise.  To prove
convergence of the solution itself (or for a broader class of functions) would
require more detailed estimates of the difference between the two as a function
of the wavenumber $k$.  Such an extension, as well
as extension to more general PDEs, is left to future work.

\subsection{Proof of Theorem \ref{thm:convergence}}
Due to linearity, it is sufficient to prove this result for the initial data
\begin{equation}
    u(x,0) = \hat{u}(0) e^{ikx} \;,
\end{equation}
The initial data \eqref{initial-values} can be written
    \begin{align}
        q_j(x,0) & = (ik)^j \uhat(0) e^{ikx}, \ \  m = 0, 1, \dots, m-1.
    \end{align}
Let $A, B, D$ be defined as in \eqref{matform} with $P$ defined by \eqref{stableP-plus}.
%    \begin{align*}
%        P = 
%        \begin{bmatrix}
%        0 & 0 & \dots  & 0 & -1 \\
%        0 & 0 & \dots  & 1 & 0 \\
%        \vdots  & \vdots & \rotatebox[origin=c]{70}{$\ddots$} & \vdots &  \vdots \\
%        0 & -1 & \dots  & 0 & 0 \\
%        1 & 0 & \dots  & 0 & 0 \\
%        \end{bmatrix}\;,
%    \end{align*}
%    where the signed permutation matrix $P$ is determined by the stability conditions of the hyperbolized system. 
Let $L_\tau = -i \tau \Lambda (kA + iB)$; i.e.
    \begin{equation} \label{Ldef}
        L_\tau =
        \begin{bmatrix}
            0 & -ik \tau \\
            -ik P & 0 \\
        \end{bmatrix} 
        +
          \begin{bmatrix}
            0 & 0 \\
            0 & P \\
        \end{bmatrix} = 
        \begin{bmatrix}
        0 & 0 & \dots  & \dots & 0 & 0 & -ik\tau \\
        0 & 0 & \dots  & \dots & 0 & ik & -1 \\
        0 & 0 & \dots & \dots &  -ik & 1 & 0\\
        0 & 0 & \dots  & ik & -1 & 0 & 0 \\
        \vdots & \vdots & \rotatebox[origin=c]{70}{$\ddots$} & \rotatebox[origin=c]{70}{$\ddots$} & \vdots & \vdots & \vdots \\
        0 & ik & -1 & 0 & \dots  &  \dots & 0  \\
        -ik & 1 & 0 & 0 & \dots &  \dots & 0 \\
        \end{bmatrix} \;.
    \end{equation}

\begin{lemma} \label{lem:L}
    The matrix $L_\tau$ has the following properties:
    \begin{enumerate}
        \item $L_0$ has an eigenvalue zero of multiplicity one, with corresponding
                eigenvector $(r_0)_j = (ik)^{j-1}$, where $1\le j\le m$.
        \item $L_0$ is diagonalizable with $m$ distinct eigenvalues, and the remaining eigenvectors are orthogonal to $r_0$.
        \item $L_\tau$ is diagonalizable, and all of its eigenvalues have zero real part, for all $\tau>0$.
        \item $L_\tau$ has an eigenvector $r_0 + \order(\tau)$, with eigenvalue $\omega_0(\tau) = -\tau(ik)^m + \order(\tau^2)$.
        \item Let $R_\tau$ denote the matrix of right eigenvectors of $L_\tau$.  Then
                $$
                    R_\tau = R_0 + \order(\tau).
                $$
    \end{enumerate}
\end{lemma}

\begin{proof}
Properties 1-2: 
Taking $\tau=0$, the first row of $L_\tau$ vanishes, so $L_0$ has a zero
eigenvalue.  The corresponding eigenvector can be found by direct computation.
Let $\hat{L}_0$ denote the matrix obtained by removing the first row and last
column of $L_0$.  The remaining eigenvalues of $L_0$ are those of $\hat{L}_0$,
which has determinant 1, and thus has no zero eigenvalues.  
Furthermore, $\hat{L}_0$ is skew-Hermitian, so it is diagonalizable with purely imaginary eigenvalues.  
To see that the eigenvalues of $\hat{L}_0$ are distinct, let $\lambda$ denote
any of its eigenvalues.  Then the submatrix obtained by removing the first row
and last column of $\hat{L}_0 - \lambda I$ is non-singular, so this eigenvalue has multiplicity one.
    
Property 3:  This follows from the fact that $L_\tau$ is similar to a skew-Hermitian matrix (see the proof of Theorem \ref{thm:stability}.

Property 4:  
    The first statement follows from continuous dependence of the eigenvector on $\tau$.
    The left eigenvector corresponding to the zero eigenvalue of
    $L_0$ is $e_1=[1,0,\dots,0]^T \in {\mathbb R}^m$.  Then \cite[Thm. IV.2.3]{stewart1990matrix}
    states that the corresponding eigenvalue of $L_\tau$ is
    $$
        \omega = \frac{e_1^T (L_\tau-L_0) r_0}{e_1^T r_0} + \order(\tau^2) = -\tau (ik)^m + \order(\tau^2).
    $$

Property 5:  
    This follows from the fact that $L_\tau$ is diagonalizable with
    distinct eigenvalues; see \cite[Thm. 8, p. 130]{lax2007linear}.  \end{proof}

Finally, we prove Theorem \ref{thm:convergence}.

\begin{proof}
    The solution of \eqref{linear2} can be written as
    $q_{j}(x,t) = \qhat_j(t)e^{ikx}$ where $\qhat$
    is given by the solution of the initial-value ODEs
    \begin{subequations} \label{relaxed-ivp}
    \begin{align*}
        \qhat'(t) & = \tau^{-1} L_\tau \qhat(t) \\
        \qhat_j(0) & = (ik)^j \uhat(0), \ \ \ j = 0, 1, 2, \dots, m-1.
    \end{align*}
    \end{subequations}
    Here $L_\tau$ is given by \eqref{Ldef}.  By Lemma \ref{lem:L} property 3, we can write $L_\tau = R_\tau \Omega_\tau R^{-1}_\tau$, ordering the eigenvalues
    to that the top-left entry of $\Omega_\tau$ is the value referred to in
    Lemma \ref{lem:L}, property 4.  Then the solution of \eqref{relaxed-ivp} is
    \begin{align*}
        \qhat(t) & = \exp(t\tau^{-1} L_\tau) \qhat(0) \\
                 & = \exp(t \tau^{-1} R_\tau \Omega_\tau R_\tau^{-1}) \qhat(0) \\
                 & = R_\tau \exp(t \tau^{-1} \Omega_\tau) R_\tau^{-1} \qhat(0).
    \end{align*}
    Note that the entries of $\exp(t\tau^{-1}\Omega_\tau)$ remain bounded
    even for $\tau\to 0$ and arbitrary $t$, due to Lemma \ref{lem:L} property 3.
    So using Lemma \ref{lem:L} property 4, we can write
    $$
        \exp(t\tau^{-1}\Omega_\tau) = \exp(-t(ik)^m + \order(\tau)) e_1 e_1^T + X
         = (\exp(-t(ik)^m)+ \order(\tau)) e_1 e_1^T + X 
    $$
    where $e_1 = [1, 0, \dots, 0]^T$ and $X$ denotes a diagonal matrix
    whose entries are bounded as $\tau\to0$ and whose top left entry
    vanishes, so $Xe_1 = 0$.
    
    Now since $R_\tau=R_0+\order(\tau)$, then $R_\tau^{-1} = R_0^{-1} + \order(\tau)$.
    Furthermore, by the definition of $\qhat(0)$ and using
    Lemma \ref{lem:L} properties 1-2, we have
    $R_0^{-1}\qhat(0) = \uhat(0) e_1$, 
    So we have
    \begin{align*}
        \qhat(t) & = R_\tau \left((\exp(-t(ik)^m)+ \order(\tau)) e_1 e_1^T + X \right) (e_1 \uhat(0) + \order(\tau)) \\
                 & = \exp(-t (ik)^m) R_\tau e_1 \uhat(0) + \order(\tau) \\
                 & = \exp(-t (ik)^m) (R_0 + \order(\tau)) e_1 \uhat(0) + \order(\tau) \\
                 & = \exp(-t (ik)^m) \qhat(0) + \order(\tau) \\
                 & = \uhat(t) + \order(\tau),
    \end{align*}
    where we used Lemma \ref{lem:L} properties 3 and 4.
\end{proof}

\section{Additional Nonlinear Examples} \label{sec:examples}
In this section we apply and extend the foregoing theory in relation to some
widely-studied nonlinear PDEs.  Whereas the examples from the introduction
follow our theory in a straightforward way, here we focus on examples that
have additional complications.

In the numerical solutions for these examples, we use a Fourier pseudospectral collocation method in
space and explicit 3rd-order Runge-Kutta integration in time.
The code to reproduce each example is available online\footnote{\url{http://github.com/ketch/hyperbolization-RR}}.

\subsection{The Nonlinear Schrodinger Equation}
Let us consider the nonlinear Schrodinger (NLS) equation
\begin{align} \label{nls}
    i \partial_t u + \partial_x^2 u + \kappa |u|^2 u & = 0,
\end{align}
with linear dispersion relation $\omega(k) = k^2 - |u_0|^2$.
We introduce $q_0 \approx u$ and $q_1 \approx \partial_x u$.  In order to obtain a 
stable hyperbolic system, in this case
we use an imaginary multiplier $i \tau$ ($\tau>0$), yielding the system
\begin{subequations} \label{nlsH}
\begin{align}
    i\partial_t q_0 + \partial_x q_1 & = - \kappa |q_0|^2 q_0 \\
    i \tau \partial_t q_1 & = \partial_x q_0 - q_1.
\end{align}
\end{subequations}
Writing this system in the form \eqref{linear-hyperbolic} shows that the characteristic
speeds are $\pm\sqrt{\tau^{-1}}$.  Linearizing about $q_0=0$ yields the dispersion relation
$$
\tau\omega^2 + \omega - k^2 = 0,
%\omega^2 + (\lambda + |u_0|^2) \omega + \lambda (|u_0|^2-k^2) = 0,
$$
so that
$$
    \omega(k) = -\frac{1}{2\tau} \pm \frac{1}{2\tau} \sqrt{1+4k^2\tau}.
$$
%This reveals the importance of the sign of the multiplier; if we change the sign of
%$\lambda$, then $\omega$ will be complex and solutions will blow up.
Expanding the square root about $k=0$ shows that one of the roots is equal to
$k^2 + \order(k^4\tau^2)$, which approximates the dispersion relation of the linear
Schrodinger equation.

Figure \ref{fig:NLS} shows a comparison between the solution of the NLS equation
\eqref{nls} and those of the hyperbolic NLS equation \eqref{nlsH}, with varying
values of $\tau$.  The initial condition is a soliton solution of \eqref{nls}:
$$
    u_0(x) = \sqrt{2 \alpha \exp(ix)} \sech(\sqrt{\alpha} x).
$$
The solution shown is computed at $t=2$.

\begin{figure}
    \center
    \includegraphics[width=6in]{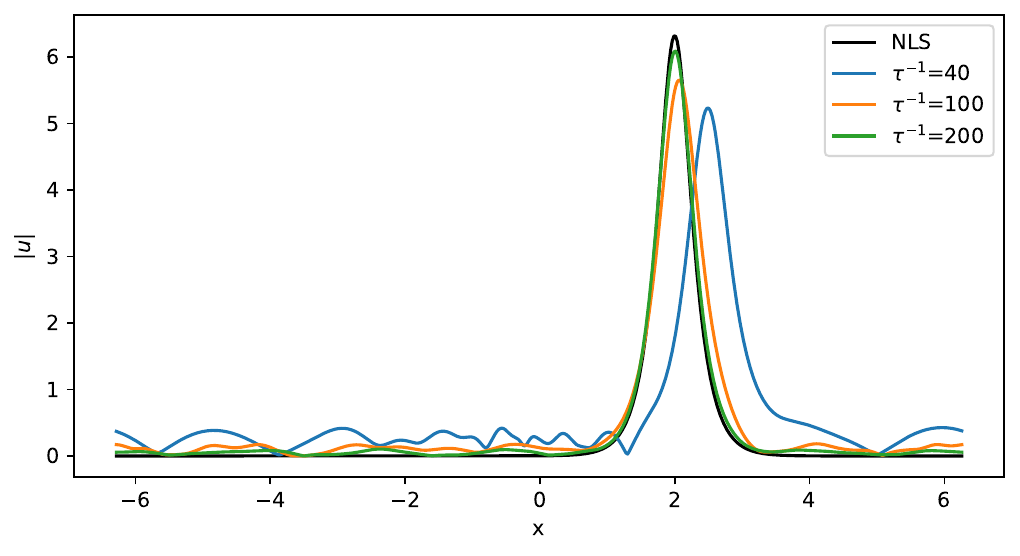}
    \caption{Comparison of the solution of the NLS equation \eqref{nls} and its hyperbolic
    approximation \eqref{nlsH}.  The approximation improves with smaller values of $\tau$.}
    \label{fig:NLS}
\end{figure}

%\subsection{Hyperbolic KdV (KdVH)}
%Let us apply the same ideas to the Korteweg-de Vries (KdV) equation:
%\begin{align}
%    \partial_t u + u\partial_x u + \partial_x^3 u & = 0.
%\end{align}
%We introduce $v \approx u_x$ and $w\approx v_x$, so that $w_x \approx u_{xxx}$.
%Next we introduce evolution equations for $v, w$ and add one of the
%"constraints" to each of them.  One possibility is
%%Rewriting \eqref{kdv} and moving algebraic terms to the right, we have the KdVH system
%\begin{subequations}
%\begin{align}
%    u_t + u u_x + w_x & = 0 \\
%    v_t + c_1 v_x & = c_1 w \\
%    w_t + c_2 u_x & = c_2 v.
%\end{align}
%\end{subequations}
%The other possibility is
%\begin{subequations}
%\begin{align}
%    u_t + u u_x + w_x & = 0 \\
%    v_t + c_1 u_x & = c_1 v \\
%    w_t + c_2 v_x & = c_2 w.
%\end{align}
%\end{subequations}
%The question immediately arises: which of these systems (if any) is preferable?  And what signs
%should be taken by $c_1, c_2$?  Regarding the first question, the homogeneous hyperbolic
%part of the first system is naturally symmetrizable.

\subsection{The Camassa-Holm Equation}
Next we consider the Camassa-Holm (CH) equation
\begin{align}
    \partial_t u - \partial_t \partial_x^2 u + 3u\partial_x u - 2\partial_x u \partial_x^2 u - u\partial_x^3 u = 0, \ \text{on} \ [x_{L},x_{R}]\times(0,T] \;,
\end{align}
with initial data $u(x,0) = \left(\frac{\pi}{2}\right) e^x - 2 \sinh(x) \arctan(e^x) - 1$ \cite{camassa1994new},  and the 
periodic boundary condition $u(x_L,t) = u(x_R,t)$. Figure 7 in \cite{camassa1994new} illustrates the solution to this problem. 
The initial parabolic pulse steepens and then breaks into a train 
of peakon solitons, which move at speeds proportional to their amplitudes. This 
example of hyperbolization is interesting in that it contains 
two different third-order derivative terms, one of which is nonlinear.
Introducing equations for $q_1$ and $q_2$ in the same way as for the KdV
equation, we write
%We introduce
%\begin{align*}
%    v & \approx \partial_x u, & w & \approx \partial_x v,
%\end{align*}
%so that 
%\begin{align*}
%    \partial_x w & \approx \partial_x^3 u, & \partial_t w & \approx \partial_t \partial_x^2 u,
%\end{align*}
%which motivates the system
\begin{align*}
    \partial_t q_0 - \partial_t q_2 + 3q_0\partial_x q_0 - 2 \partial_x q_0\ \partial_x q_1\  - q_0\partial_x q_2 & = 0 \\
    \tau \partial_t q_1 + (\partial_x q_1 - q_2) & = 0\\
    \tau \partial_t q_2 - (\partial_x q_0 - q_1) & = 0.
\end{align*}
This system is not in the usual hyperbolic form, since the first equation contains time derivatives of both $q_0$ and $q_2$,
and the term $2 \partial_x q_0 \partial_x q_1$, which is a product of two first-order partial derivatives. 
The time derivative $\partial_t q_2$ can be eliminated by adding $\tau^{-1}$ times the last equation to the first equation. 
Additionally, the term $2 \partial_x q_0 \partial_x q_1$ can be replaced by either $2 q_1 \partial_x q_1$ or $2 q_2 \partial_x q_0$. 
We chose the latter, which leads to the hyperbolized CH (CHH) system
\begin{subequations}
\begin{align}
    \partial_t q_0  + 3q_0\partial_x q_0 - 2 q_2 \partial_x q_0  - q_0\partial_x q_2 - \tau^{-1} (\partial_x q_0 - q_1) & = 0 \\
    \tau \partial_t q_1 + (\partial_x q_1 - q_2) & = 0\\
    \tau \partial_t q_2 - (\partial_x q_0 - q_1) & = 0.
\end{align}
\end{subequations}

\begin{figure}
    \center
    \includegraphics[width=6in]{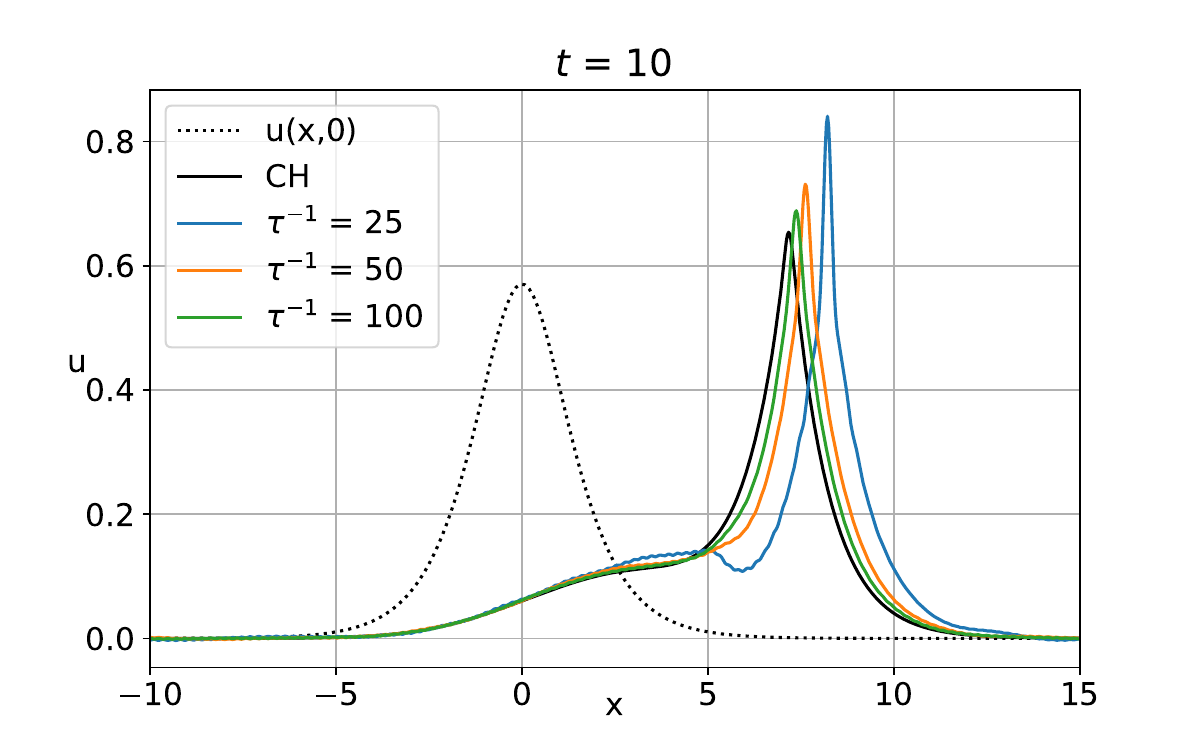}
    \caption{Comparison of solution of the Cammasa-Holm equation and its hyperbolic approximation.
    The approximation improves with smaller values of $\tau$.}
    \label{fig:CH}
\end{figure}

Interestingly, the linear dispersion relation for this system is stable
for both positive and negative values of $\tau$.
We choose $\tau^{-1} = 25, 50, 100$ and solve the corresponding hyperbolized systems on the domain $[x_L, x_R] = [-10, 50]$ 
using a pseudo-spectral semi-discretization in space with $m=512$ spatial grid points and the SSPRK33 method in time, with $t=10$. 
The solution of the original CH equation is also computed using the pseudo-spectral semi-discretization in space and the SSPRK33
method in time, employing the same number of grid points in space. Figure~\ref{fig:CH} illustrates that the solution of the CHH 
tends to the solution of the original CH equation as the hyperbolization parameter $\tau$ decreases.

%\section{Questions to be addressed}
%\begin{itemize}
%    \item What is the stable step size?  At least a precise answer for linear systems.
%    \item Is the hyperbolization approach more efficient computationally?  How do the costs scale?
%    \begin{itemize}
%        \item Stiffness goes like some power of $\Delta x^{-1}$; truncation error goes like some power of $\Delta x$.
%        \item Hyperbolization error goes like what?
%    \end{itemize}
%\end{itemize}

\subsection{The Kuramoto-Sivashinsky Equation}
In this section we hyperbolize the Kuramoto-Sivashinsky (KS) equation
\begin{align} \label{KS}
    \partial_t u + \partial_x^2 u + \partial_x^4 u + u \partial_x u & = 0.
\end{align}
To our knowledge, this is the first example of hyperbolization of an equation with derivatives of order 
greater than three.  The linear dispersion
relation for \eqref{KS} is
$$
    \omega(k) = k u_0 + i (k^2 - k^4).
$$
Note that for small wavenumbers $k<1$, the linearized KS equation is unstable;
only the presence of the nonlinear term prevents blowup of solutions.  It is
natural to ask whether this feature prevents the generation of a stable
hyperbolization.

Based on the analysis in Section \ref{sec:stability}, we propose the hyperbolic
system
\begin{subequations} \label{KSh}
\begin{align}
    \partial_t q_0 + \partial_x q_1 + \partial_x q_3 + q_0 \partial_x q_0 & = 0 \\
    \tau \partial_t q_1 - (\partial_x q_2 - q_3) & = 0 \\
    \tau \partial_t q_2 - (\partial_x q_1 - q_2) & = 0 \\
    \tau \partial_t q_3 + (\partial_x q_0 - q_1) & = 0,
\end{align}
\end{subequations}
with $\tau >0$.  Solutions of \eqref{KS} are known to exhibit both periodic behavior
and chaotic behavior, depending on the problem setup.  A preliminary study of solutions
of \eqref{KSh} shows that it behaves similarly, yielding periodic or chaotic solutions
under similar circumstances.  The solutions of \eqref{KSh} are observed
to converge to those of \eqref{KS} as $\tau \to 0$, when the solution is not
chaotic.  
%Under conditions in which \eqref{KS} becomes chaotic, that of \eqref{KSh}
%seems to do the same, and the time to onset of chaos is nearly the same. \abhi{We decided to remove this comment.}
% As one might expect, the chaotic solutions do not remain close to each other.  An
% example is shown in Figure \ref{fig:KS}, where the solution
% is computed using a Fourier pseudospectral method with 128 points in space.
% We note that the breakup of the initial sinusoidal condition occurs at very
% nearly the same time, although the solutions afterward are completely different.
%The solution of the hyperbolized equation seems, if anything, even more chaotic
%than that of the original KS equation.

% \begin{figure} 
%     \center
%     \begin{subfigure}{0.45\textwidth}
%     \centering
%     \includegraphics[width=\textwidth]{figures/KS.pdf}
%     \caption{KS equation}
%     \end{subfigure}
%     \hfill
%     \begin{subfigure}{0.45\textwidth}
%     \centering
%     \includegraphics[width=\textwidth]{figures/KSh.pdf}
%     \caption{Hyperbolic KS equation with $1/\tau = 200$}
%     \end{subfigure}
%     \caption{Comparison of the solution of the KS equation \eqref{KS} and its hyperbolic
%     approximation \eqref{KSh}, for a case in which the solution is chaotic.  Here $\tau=1/200$.}
%     \label{fig:KS}
% \end{figure}

\begin{figure}
    \center
    \includegraphics[width=4in]{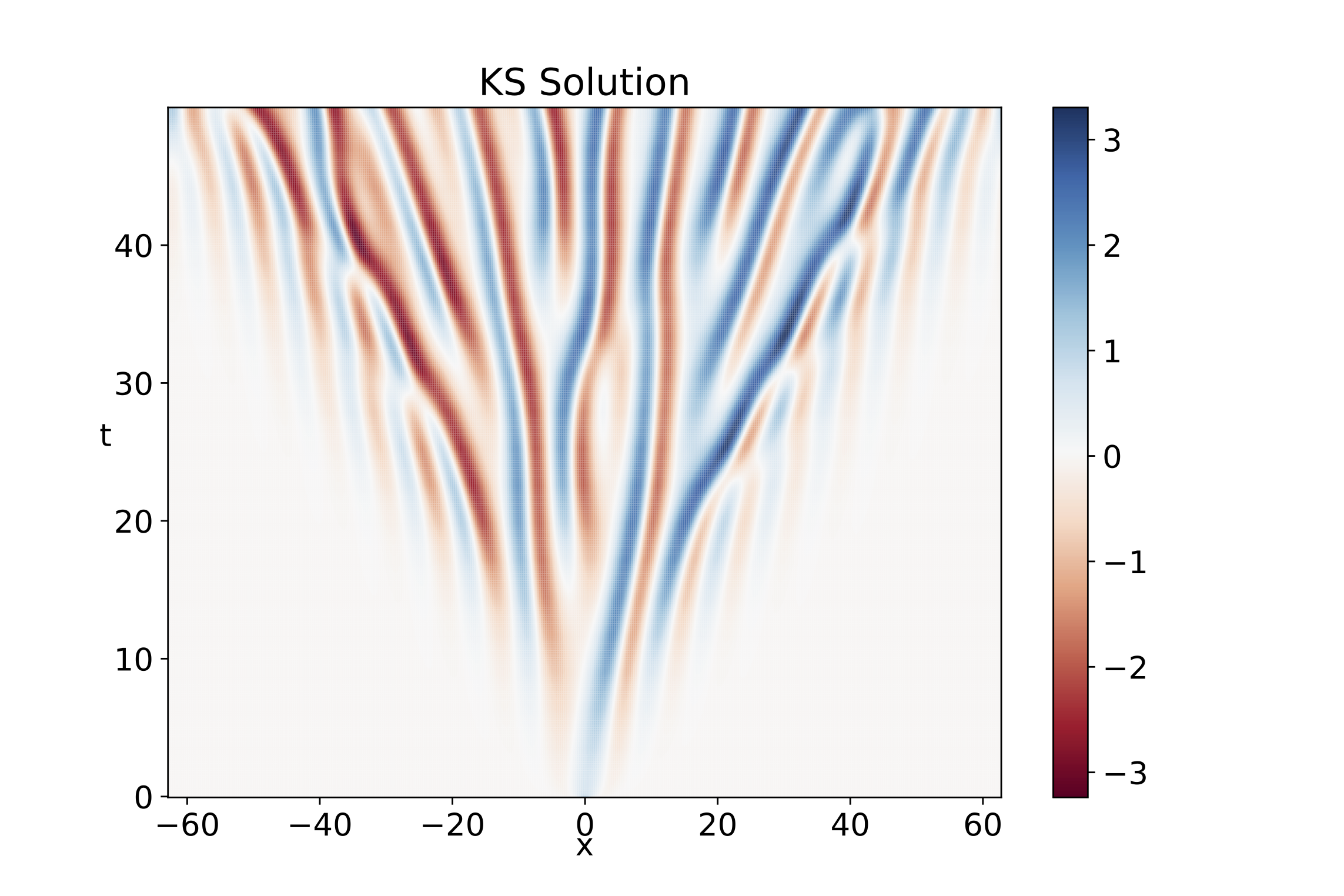}
    \caption{Solution of the KS equation \eqref{KS} up to $t=50$.}
    \label{fig:KS_Sol}
\end{figure}

\begin{figure}[htb]
    \begin{minipage}[b]{.48\textwidth}
    \includegraphics[width=\textwidth]{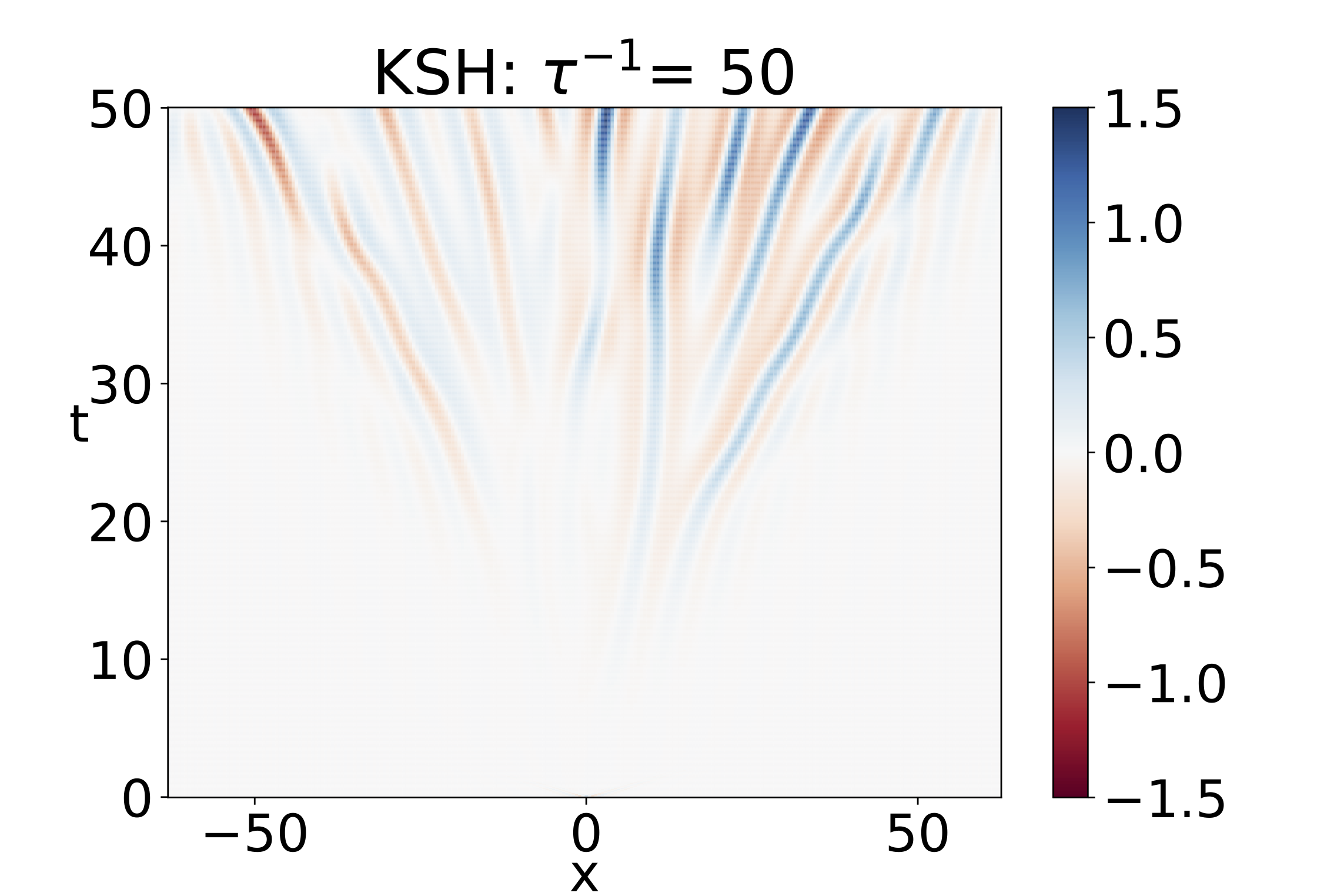}
    \end{minipage}
    %\hfill
    \begin{minipage}[b]{.48\textwidth}
    \includegraphics[width=\textwidth]{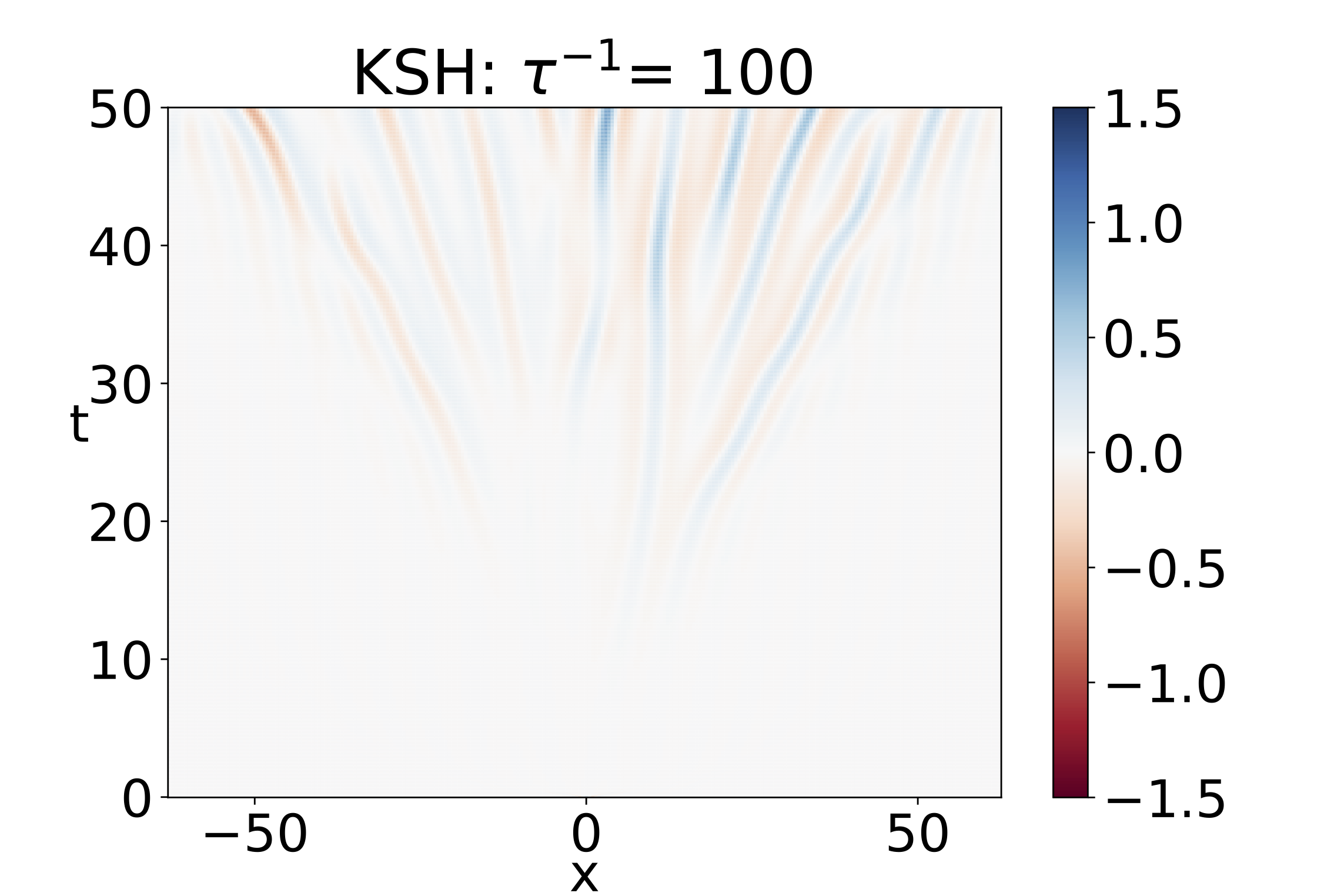}
    \end{minipage}
        %\hfill
    \begin{minipage}[b]{.48\textwidth}
    \includegraphics[width=\textwidth]{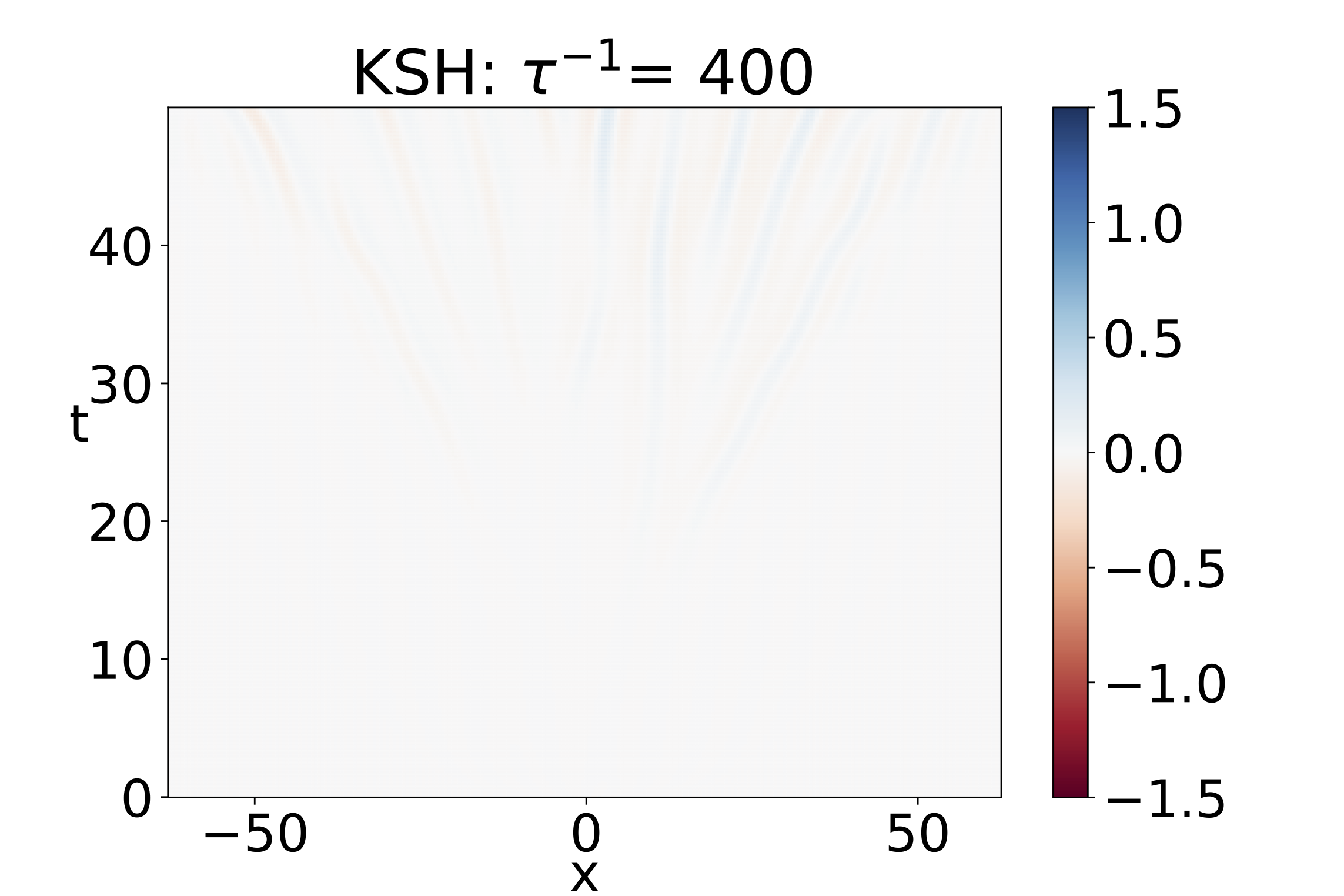}
    \end{minipage}
        %\hfill
    \begin{minipage}[b]{.48\textwidth}
    \includegraphics[width=\textwidth]{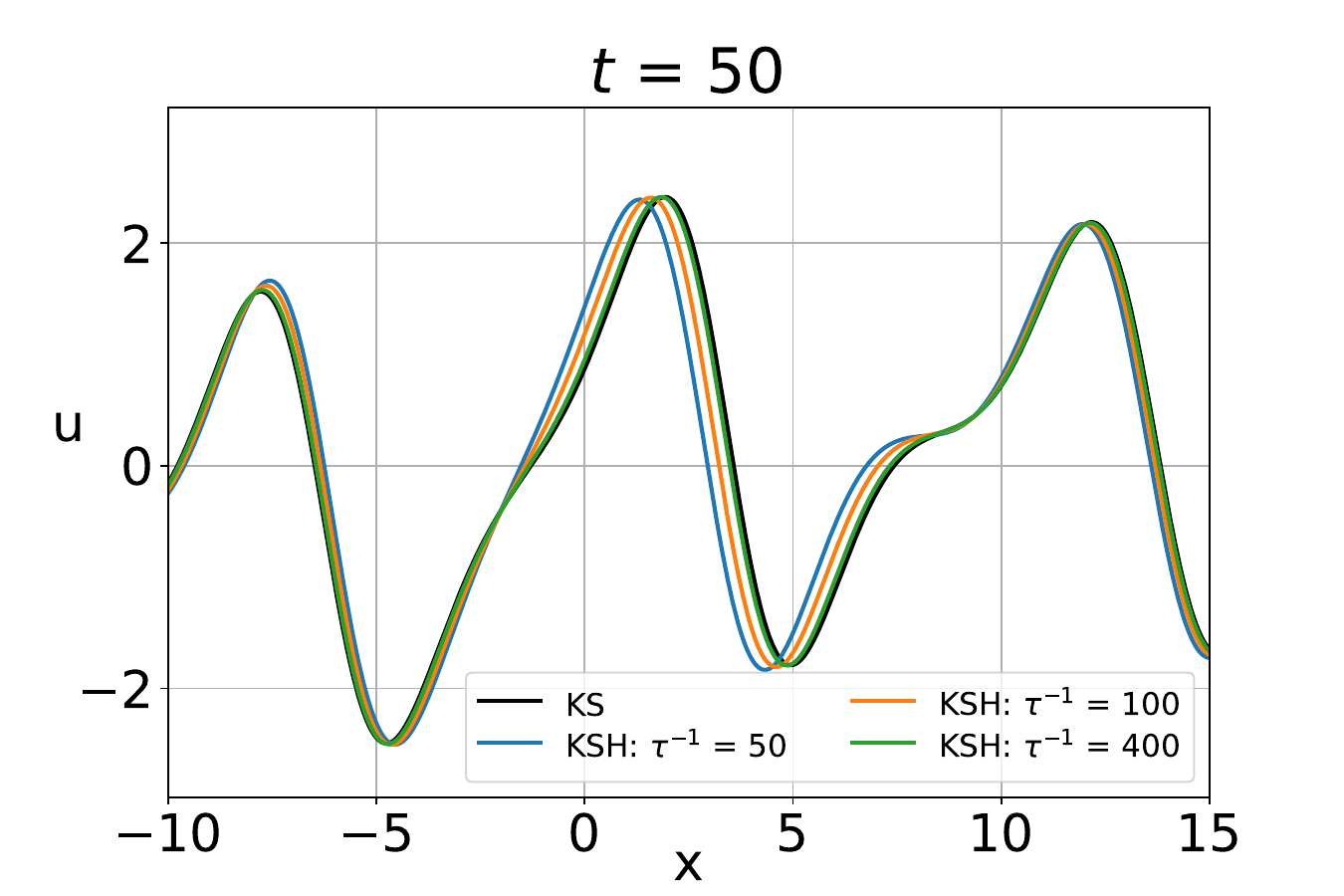}
    \end{minipage}
	\caption{Point-wise hyperbolization error of the solution of the hyperbolic approximation \eqref{KSh} to the KS equation \eqref{KS} 
    with a Gaussian initial condition for three different relaxation parameters. Top left panel: KSH with $\tau^{-1}=50$. Top right panel: 
    KSH with $\tau^{-1}=100$. Bottom left panel: KSH with $\tau^{-1}=400$. Bottom right panel: Comparison of solution of the KS equation
     and its hyperbolic approximation at $t=50$.}
	\label{fig:KS}
\end{figure} 
Here we present an example demonstrating the convergence of the solution of the hyperbolized Kuramoto-Sivashinsky (KSH) 
equation to the solution of the original Kuramoto-Sivashinsky (KS) equation as $\tau$ tends to 0. 
We consider the problem on the domain $[-20\pi, 20\pi]$ with the Gaussian initial condition $u(x,0) = e^{-x^2}$ and periodic boundary conditions. 
Figure \ref{fig:KS_Sol} shows 
the solution to the KS equation \eqref{KS} with a Gaussian initial condition, up to time $t=50$. Each KSH system with different 
relaxation parameters is solved using the pseudospectral method with $256$ grid points in space and the classical 4th-order RK 
method in time, up to a final time of $t=50$. The corresponding solution of the original KS equation at each time is also obtained 
using the pseudospectral method in space and a 4th-order ImEx method in time, with the same resolution in space and time. 
Figure~\ref{fig:KS} shows the point-wise errors of the solution of the KSH equation compared to the KS equation \eqref{KS}, 
and we observe that the solution of the KSH equation converges to the solution of the KS equation as $\tau$ tends to 0. The 
convergence of the solution of the KSH to the solution of the KS equation at $t=50$ is shown in the bottom right panel.

\section{Discussion} \label{sec:conclusion}
PDEs are perhaps the most widely used form of mathematical model, and
herein we have shown that a very wide range of PDEs can be approximately
transformed into a completely different class of PDE, potentially enabling their
solution and analysis by different techniques than what have so far
been applied to them.

While the first work on hyperbolization of high-order PDEs dates back to the 1950s,
development in this area has accelerated in the last 15 years, during which
a variety of techniques in this vein have been applied to a number of different models.
Some of these recent works can be seen as particular cases of the general technique discussed
herein.  Awareness and understanding of hyperbolization as a general tool is likely to
facilitate its application to an increasing number of models.

Herein we have shown for the first time that a stable hyperbolization exists for
PDEs of arbitrarily high order.  Our analysis has emphasize a particular family
of linear PDEs, but the resulting approach provides a clear and systematic way
to hyperbolize any scalar evolution equation (linear or nonlinear), and we have
found this approach to be effective for every equation to which we have applied it.
We have focused on scalar evolution equations,
in order to narrow the scope enough to facilitate development of a
comprehensive theory.  But a wide variety
of other PDE models -- including elliptic problems, evolution problems with hyperbolic
constraints, and general systems of high-order PDEs -- are susceptible to a similar treatment.

Even for scalar evolution equations, there are many interesting open questions
and avenues for further research in this area, including:
\begin{itemize}
    \item General design and analysis of energy conservation or other structural properties in hyperbolized systems;
    \item development of efficient numerical discretizations;
    \item more detailed understanding of the relative efficiency of hyperbolized approximations.
\end{itemize}
Furthermore, the construction of the hyperbolized equations admits additional freedom
that has been explored in the literature only in the context of specific models; for instance
the inclusion of additional convective terms in the auxiliary equations, or the use
of multiple different relaxation times $\tau_1, \tau_2, \dots$.

It is also worth considering the possibility of gaining insight into higher-order PDE models and
solutions by analyzing their hyperbolized counterparts.  An interesting corollary of our results
is that higher-order PDE models that violate local causality can be approximated arbitrarily
well by models that respect causality.  

%Since hyperbolic PDEs admit similarity
%solutions (Riemann solutions) that can often be solved by semi-analytical means, the study of
%the Riemann problem for hyperbolized systems might yield novel insights. \abhi{We may want to remove this comment.}
%For one existing study of this kind, see \cite{grosso2010riemann}.

\bibliographystyle{plain}
\bibliography{refs}

\end{document}